\documentclass[11pt]{amsart}

\usepackage{amsgen,amsthm,amstext,amsbsy,amsopn,amsfonts,amssymb, enumerate}

\usepackage{amsmath}

\newcommand\la{\langle}
\newcommand\ra{\rangle}

\renewcommand\aa{{\mathfrak a}}

\newcommand\hh{{\mathfrak h}}

\newcommand\nn{{\mathfrak n}}

\newcommand\sso{{\mathfrak{so}}}
\newcommand\vv{{\mathfrak v}}

\newcommand\ww{{\mathfrak w}}
\newcommand\zz{{\mathfrak z}}

\newcommand\iso{{\mathfrak{iso}}}

\newcommand\NN{\mathbb N}

\newcommand\RR{\mathbb R}
\newcommand\ZZ{\mathbb Z}

\newcommand\ad{\operatorname{ad}}
\newcommand\Ad{\operatorname{Ad}}

\newcommand\Iso{\operatorname{Iso}}

\newcommand\grad{\operatorname{grad}}

\DeclareMathOperator{\En}{E}

\theoremstyle{plain}

\newtheorem{thm}{Theorem}[section]
\newtheorem{lem}[thm]{Lemma}
\newtheorem{prop}[thm]{Proposition}

\theoremstyle{definition}
\newtheorem{defn}[thm]{Definition}
\newtheorem{rem}[thm]{Remark}
\newtheorem{example}[thm]{Example}

\begin{document}

\title[first integrals on nilpotent Lie groups]
{First integrals on step-two and step-three nilpotent Lie groups}

\author{Gabriela P. Ovando}

\thanks{{\it (2000) Mathematics Subject Classification}: 70G65, 70H05, 70H06, 22E70, 22E25 }

\thanks{{\it Key words and phrases}:  }

\thanks{Partially supported by  SCyT (UNR)}

\address{CONICET- UNR, Departamento de Matem\'atica, ECEN - FCEIA, Pellegrini 250, 2000 Rosario, Santa Fe, Argentina.}

\

\email{gabriela@fceia.unr.edu.ar}


\begin{abstract} The goal of this paper is the study of algebraic relations on the Lie algebra of first integrals of the geodesic flow on  nilpotent Lie groups equipped with a left-invariant metric. It is proved  that the isometry algebra of the $k$-step nilpotent Lie group, $k=2,3$, gives rise to a isomorphic family of first integrals for the geodesic flow. Also invariant first integrals are analyzed and new involution conditions are shown. Finally it is proved that in low dimensions  complete families of first integrals can be constructed  with Killing vector fields and symmetric Killing 2-tensor fields. This holds for k-step nilpotent Lie algebras of dimension $m\leq 5$ and $k=2,3$. The situation in dimension six is also studied. 
\end{abstract}

\maketitle

 \noindent\section{Introduction}

The study of the geodesic flow is a classical topic in geometry, where Liouville integrability and involution conditions for first integrals are main questions. Indeed the existence of many examples is needed to verify any advance in the topic. In this sense compact manifolds with completely integrable geodesic flows seem to be complicated to be acquired in view of topological obstructions for the existence of analytic first integrals (see for instance \cite{Koz, Ta1, Ta2}). 

As in many other situations, Lie groups constitute a source for the construction of  several kind of examples. Different authors  studied dynamical aspects of geodesic flows on nil- and solv-manifolds (see \cite{BoT,Bu1}), manifolds which are locally homogeneous and  obtained as compact quotients of solvable or nilpotent Lie groups. Other relations between  geometrical and  dynamical aspects on solvable or nilpotent Lie groups (or their quotients), were treated for instance in  the references \cite{DM, DDM,Eb1,Eb2,GM,LW,Py,Sc}.

Indeed a first approach when dealing with spaces obtained via Lie group actions, is the study of the situation at the Lie group when equipped with a left-invariant metric. Examples of Hamiltonian systems constructed with algebraic data were given for instance in \cite{Ko2,Sy,Th} on semisimple or reductive Lie groups. However for nilpotent or solvable Lie groups, tools and  techniques are quite different (see for instance \cite{BT,KOR, Ov1}). A general goal in these cases is to reveal some algebraic relations which can be associated to the underlying geometry.  

In this work we study  algebraic relations on the Lie algebra of first integrals, when compared with the Lie algebra of the isometry group on  a Lie group $N$  which is  equipped with a left-invariant metric. We mainly concentrate on 2- and 3- step nilpotent Lie groups  and we focus on invariant first integrals and first integrals derived from Killing vector fields. Killing vector fields were generalized to symmetric Killing  tensors, which  define first integrals of the equation of motion. Results on symmetric Killing tensors can be found for instance in \cite{HMS,Se} and in particular for 2-step nilpotent Lie groups in \cite{BM}. 

The main theorems  in this work prove the statement below:

\begin{itemize}
\item Let $N$ denote a  $k$-step nilpotent Lie group, with $k=2$ or $k=3$, equipped with a left-invariant metric. Let $\iso(\nn)$ denote the Lie algebra of the isometry group. {\em  There is a monomorphism of Lie algebras from $\iso(\nn)$ to its image in $C^{\infty}(TN)$, sending $Y^* \to f_{Y^*}$ for any Killing vector field $Y^* \in \iso(\nn)$.} 

 Recall that the  Lie algebra $\iso(\nn)$  admits a decomposition as a direct sum of vector spaces $\iso(\nn)=\hh \oplus\nn$, where $\hh$ consists of skew symmetric derivations, while $\nn$ is the Lie algebra of the subgroup of isometries given by translations on the left $L_n$, $n\in N$.
 
\item The previous result is efficient for the construction of complete sets of functions in involution. This is proved in dimension five and in most cases of dimension six on k-step nilpotent Lie algebras, $k=2,3$. 

\end{itemize}

A particular section is devoted to invariant first integrals, that is, functions on $C^{\infty}(TN)$ which are invariant by the action of the Lie group $N$. One gets nice involution conditions: for instance an invariant function $g$ always Poisson commute with any function $f_{X^*}$ for any $X\in \nn$. Moreover these kind of first integrals  can be always induced to quotients of the form $\Gamma \backslash N$, for any cocompact lattice $\Gamma< N$.

 Among invariant functions, the simplest ones are  linear or quadratic polynomials in the coordinates. Precisely, quadratic polynomials are obtained from symmetric linear maps on the Lie algebra and they also can be seen as symmetric Killing 2-tensors.  We  give explicit conditions for a symmetric map $S:\nn \to \nn$ to induce a  first integral, and so, we extend to  3-step nilpotent Lie groups, results in \cite{BM} relative to symmetric Killing 2-tensors. But the goal here has no relation with the question of decomposition of the Killing tensors. It seems that until now, 3-step nilpotent Lie groups were little considered for the topic of the present paper, so that a main contribution here is to open the study to $k$-step nilpotent Lie groups for finding new examples of completely integrable geodesic flows.

The   next section consists of  preliminaries about  the symplectic structure and its corresponding Poisson bracket on the tangent space $TN$. The main reason for consider $TN$ instead of the cotangent space is to make easier  the reading of   relations between the geometry and the algebraic structure of the corresponding Lie group $N$, when equipped with a left-invariant metric. This will be seen along the text.


\section{The geodesic flow on Lie groups}

In this section we provide general notions for the study of the geodesic flow on Lie groups when equipped with a metric that is invariant by translations on the left.

 Let $(M, \la \,, \, \ra)$ denote a Riemannian manifold. Its cotangent bundle has a symplectic structure which, via the metric, can be transported to the tangent bundle $TM$. For a given differentiable function $f:TM \to \RR$, there exists a vector field $X_f$ called the Hamiltonian vector field for $f$,  implicitly defined with help of the symplectic structure $\Omega$:
$$df(v)=\Omega(X_f, v) \qquad \quad \mbox{ for all } v. $$
Making use of the metric one also has the gradient field of $f$, defined by
$$df(v)=\la \grad\,f, v\ra \qquad \quad \mbox{ for all } v. $$
We restrict the attention to Lie groups endowed with left-invariant metrics. 
Let $N$ denote a Lie group with Lie algebra $\nn$ and let $\la,\,,\,\ra$ denote a metric on  $\nn$ which  is induced to $N$ so that  translations on the left are isometries. This is called a left-invariant metric on $N$.  

The tangent bundle $TN$ is trivialized so that it is isomorphic to  $N \times \nn$. Here we identify 
$$v_p\in TN \longleftrightarrow (p, (dL_{p^{-1}})_p v_p),$$ where the last pair, is sometimes called the body coordinates for $v_p$ (see \cite{A-M}). The metric on $\nn$ is induced to $TN$ as a product metric,  in fact $T_{(p,Y)}TN\simeq \nn\times \nn$ for all $p\in N, Y\in \nn$. 

The symplectic structure on the tangent bundle $TN$ induces a Poisson bracket on the space of differentiable functions $C^{\infty}(TN)$: given  functions $f,g\in C^{\infty}(TN)$, its Poisson bracket follows
\begin{equation}\label{poisson2}\{f, g\}(p,Y)=\Omega(X_f, X_g)=\la U, V'\ra - \la U', V\ra - \la Y, [V, V']\ra, 
\end{equation}
where we denote by $\grad f(p,Y)=(U,V)$ and $\grad g(p,Y)=(U',V')$, the respective gradient fields of $f$ and $g$.  

Recall that the energy function $\En: TN \to \mathbb R$  given by 
$$\En(p, Y) = \frac{1}{2}\langle Y, Y\rangle,$$
 has a  Hamiltonian vector field $X_{\En}$ whose flow is  called the {\em geodesic flow}. It is known that if $c(t)$ is an integral curve of $X_{\En}$ then $\pi\circ c$ is a geodesic on $N$, where $\pi:TN \to N$ denotes the usual projection (see for instance \cite{A-M}). 

We say that a pair of functions $f,g\in C^{\infty}(TN)$ are {\em in involution or Poisson commute}   whenever $\{f, g\}\equiv 0$.

The definitions above imply that  $\{f,g\}=0$ if and only if $df(X_g)=0$  if and only if $X_g(f)=0$ (and also $X_f(g)=0$). That is $f$ is constant along integral curves of $X_g$ (analogously for $X_f$). 

In particular by fixing the energy function $\En$, any function $f$ which Poisson commutes with $\En$ is called a {\em first integral} of the geodesic flow. Note that for any left-invariant metric on $N$, the gradient has the expression $\grad \En(p,Y)=(0,Y)$ for all $p\in N, Y\in \nn$. Thus, the equality $\{f,\En\}=0$ is equivalent to 
\begin{equation}\label{first-int}
\la Y,U \ra = \la Y, [V, Y]\ra \qquad \quad \mbox{ for all } Y\in \nn,
\end{equation}
  where the gradient of $f\in C^{\infty}(TN)$ is denoted  by $\grad f(p,Y)=(U,V)$.

\begin{defn}  Let ($N,\la\,,\ra)$ be a Riemannian manifold of dimension $n$. The geodesic flow is \emph{completely integrable} (in the sense of Liouville) if there exist $n$ first integrals of the geodesic flow, $f_i: TM \to \RR$, such that they pairwise Poisson commute $\{f_i, f_j\} = 0$ for all $i, j$ and the gradients of $f_1, \ldots, f_n$ are linear independent on an open dense subset of $TN$.
\end{defn}

Indeed the first step in the study of integrability is to get first integrals. There are several methods to obtain them.  It is known that whenever $M$ is a Riemannian manifold and $X^*$ is a Killing vector  field on $M$,  the function $f_{X^*}: TM \to \mathbb R$ defined as 
$$f_{X^*}(v) = \langle X^*(\pi(v)), v\rangle$$
 is a first integral of the geodesic flow. Recall that the set of Killing vector fields builds a Lie algebra, when considered with the Lie bracket,  that is isomorphic to the Lie algebra of the isometry group of $M$. In the next paragraphs we focus on this situation for Lie groups when equipped with left-invariant metrics. 

\begin{example} \label{heis} The Heisenberg Lie group of dimension three can be presented as the set of real matrices of the form 
$$\left( \begin{matrix}
1 & x_1 & z \\
0 & 1 & y_1 \\
0 & 0 & 1
\end{matrix}
\right)
$$
which is a three dimensional Lie group when considered with the usual matrix multiplication. With usual coordinates in $\RR^3$ one has 
$$(x_1,y_1,z)(x_1',y_1',z')=(x_1+x_1', y_1+y_1', z+z'+x_1y_1'), \qquad \mbox{ for } x_1,x_1', x_2, x_2', z, z' \in \RR.
$$
One denotes this group by $H_3$. Its Lie algebra $\hh_3$ has a basis of left-invariant vector fields given by:
$$
X_1(x_1,y_1,z)  = \partial_{x_1} \quad
Y_1(x_1,y_1,z)  =  \partial_{y_1} + x_1 \partial_z \quad
Z(x_1,y_1,z)  =  \partial_z,
$$
where the next  notation is used: $\partial_{\nu}=\frac{\partial}{\partial \nu}$. This basis satisfies the non-trivial Lie bracket relation:
$[X_1,Y_1] = Z$.

 The following  is a linear independent set of right-invariant vector fields at every $(x,y,z)\in \RR^3$:
$$X_1^*(x_1,y_1,z)= X_1 + y_1 Z, \quad Y_1^*(x_1,y_1,z) = Y_1 - x_1 Z, \quad Z^*(x_1,y_1,z)=Z.
$$
Take the metric on $H_3$ making of $X_1,Y_1,Z$ a orthonormal basis. On $\RR^3$ this is given by
$$g_{(x_1,y_1,z)}= dx_1^2+ (1+x_1^2) dy_1^2 + dz^2 - x_1 dy_1 dz.$$
Set the product metric on $TH_3 \simeq H_3 \times \hh_3$.  Since right-invariant vector fields correspond to Killing vector fields,  one has the corresponding first integrals $f_{X_1^*}, f_{Y_1^*}, f_Z$. Moreover since $\aa:=span\{Y_1,Z\}$ is a abelian ideal of codimension one, the geodesic flow on $H_3$ is completely integrable,  actually for any left-invariant metric on $H_3$ (as a consequence of Theorem \ref{t2} below). 

The Heisenberg Lie group of dimension $2n+1$, $H_{2n+1}$, is constructed with underlying manifold $\RR^{2n+1}$ and by generalizing the previous multiplication

\smallskip

$(x_1, \hdots, x_n, y_1, \hdots, y_n, z) (x_1', \hdots, x_n', y_1', \hdots, y_n', z) = $

\qquad \qquad \qquad \qquad  $= (x_1+x_1', \hdots, x_n+x_n', y_1+y_1', \hdots, y_n+y_n', z+z'+\sum_{i=1}^n x_i y_i').$

\smallskip

One gets non-trivial Lie brackets $[X_i, Y_i]=Z$ for $i=1, \hdots n$.

The integrability of the geodesic flow on compact quotients $\Gamma \backslash H_{2n+1}$ for families of lattices $\Gamma < H_{2n+1}$, was studied in \cite{Bu1,KOR}. 
\end{example}

A  nilpotent Lie group $N$  is a Lie group whose Lie algebra is nilpotent: it is $k$-step whenever the  $k$ Lie brackets are zero: $[x_1,[x_2,[\hdots, x_{k+1}]]]=0$. Denote by   $\nn$ the corresponding Lie algebra, thus for  
\begin{itemize}
	\item 2-step nilpotent one has $[U,[V,W]]=0$ for all $U,V,W\in \nn$;
	\item 3 -step nilpotent one has $[X,[U,[V,W]]]=0$.
\end{itemize}

And the $k-1$ Lie brackets are not zero. For $j\geq 1$, denote by $C^j(\nn)=[\nn, C^{j-1}(\nn)]$ with $C^0(\nn)=\nn$, which is called the central descending series of $\nn$. Note that if $\nn$ is $k$-step nilpotent then $C^{k-1}(\nn)\subseteq \zz(\nn)$ where $\zz(\nn)$ is the center of $\nn$. The commutator of the Lie algebra is the ideal defined by the subspace spanned by $[\nn,\nn]$. 

\begin{example} \label{n2} In dimension four one finds a 3-step nilpotent Lie algebra denoted by $n_2$. It has a basis $e_1, e_2, e_3, e_4$ satisfying the Lie bracket relations $[e_1, e_2]=e_3, \quad [e_1, e_3]=e_4$. 
	
	In this case the commutator is spanned by $e_3$ and $e_4$ belonging $e_4$ to the center of $n_2$. 
	
	Notice that the commutator of any 3-step nilpotent Lie algebra is  abelian. 
	
\end{example}

The Backer-Campbell-Hausdorff  (see for instance \cite{Va}), formula gives  the next expressions for the product on $N$:
\begin{itemize}
	\item for 2-step nilpotent:  $\exp(W) \exp(U)= \exp(W + U +\frac12 [W,U])$;
	\item for 3 -step nilpotent: $\exp(W) \exp(U)= \exp(W + U +\frac12 [W,U] +\frac1{12} ([W,[W,U]] + [U,[U,W]]))$.
\end{itemize}

Let $N$ be a nilpotent connected and simply connected Lie group. Thus if $N$ has dimension $n$, it can be modeled on $\RR^n$. Moreover,  the exponential map is a diffeomorphism from $\nn$ to $N$,  that sends the element $V=\sum x_i X_i$ $\mapsto (x_1, x_2, \hdots, x_n)$, in  the usual coordinates in $\RR^n$, where $\{X_i\}$ is a basis of 
(left-invariant) vectors in $\nn$. 

\section{the isometry algebra  and first integrals} 
The goal of this section is  the study of first integrals associated to the isometry group of a $k$-step nilpotent Lie group $N$, for $k=2,3$. We get explicit formulas and prove one of the main results of the paper, which will be applied later.   

Provide $N$ with a Riemannian metric $\la\,,\,\ra$ which is invariant by translations on the left. Via the identification $\exp(X)=X$,  
the group of orthogonal automorphisms of $N$  is in correspondence to the group $H$ consisting of linear maps $\psi:\nn \to \nn$ satisfying:
\begin{enumerate}[(i)]
	\item $\la \psi(X), \psi(Y) \ra = \la X, Y \ra$ for all $X, Y \in \nn$;
	\item $ \psi [X, Y] = [\psi(X), \psi(Y)]$ for all $X, Y \in \nn$.
\end{enumerate}

Moreover, in this situation that the metric on $N$ is  left-invariant, the isometry group is the semidirect product
$$\Iso(N)=H \ltimes N, \qquad\mbox{ where }{\begin{array}{ll}
	N &\mbox{is the nilradical of } \Iso(N) \mbox{ and } \\
	H  & \mbox{is the  group of orthogonal automorphisms,}
	\end{array}}$$
result that was proved in 1963 by Wolf  \cite{Wo} (see also \cite{Wi}). The action of $H$ on $N$ is given by $f \cdot L_n = L_{f(n)}$ for every left translation $L_n$, for $n\in N$, and every $f\in H$.

The corresponding Lie algebra, denoted by $\iso(\nn)$, is the direct sum as vector spaces $\iso(\nn)=\hh \ltimes \nn$, where $\hh$ is the Lie  subalgebra of $H$ and $\nn$ an ideal being the Lie algebra of $N$. The  Lie subalgebra $\hh$  consists of skew-symmetric derivations,
$$\hh=\{ D : \nn \to \nn \, : \, D[X, Y]=[DX, Y] + [X, DY] \mbox{ and } \la DX, Y\ra + \la X, DY\ra =0\}.$$

Also any vector  $X\in T_eN$ gives rise to a right-invariant vector field  on $N$ by
$$X^*(p) = \frac{d}{ds}_{|_{s=0}} \exp(sX) p \qquad \mbox{ for all }p\in N,  $$ 
Clearly $\exp(sX)p$ is a one-parameter group of isometries corresponding to the translation on the left by $\exp(sX)$,  $L_{\exp(sX)}$. 
Notice that since $\exp(sX) . p = p . p^{-1} \exp(s X)p$ one has
$X^*(p) = \frac{d}{ds}|_{s=0} \exp(sX) p = dL_p \Ad(p^{-1}) X$, which is the Killing vector field associated to $X\in   \nn$. By making use of this,   the differentiable function $f_{X^*}:TN \to \RR$ is given by
\begin{equation}\label{killingf}
f_{X^*}(p,Y)  =  \la \Ad(p^{-1}) X, Y\ra.  
\end{equation}
The computation of  the derivative
$$\frac{d}{ds}_{|_{s=0}}f_{X^*}(p\exp sU), Y + s V)$$
brings the  following expression for the gradient vector field
\begin{equation}\label{gradleft}
\grad f_{X^*}(p,Y)= (\ad^{\tau}(\Ad(p^{-1})X)(Y), \Ad(p^{-1})X),
\end{equation}
where $\ad^{\tau}(X)$ denotes the transpose of $\ad(X)$ relative to the metric on $\nn$,  $\la\,,\,\ra$. In view of this one verifies Equation \eqref{first-int} 
so that any $f_{X^*}$ is a first integrals of the geodesic flow.

Assume that $X_1, X_2\in\nn$ so that  $X_1^*$ and $X_2^*$ are the corresponding right-invariant vector fields  with respective functions $f_{X_1^*}, f_{X_2^*}: TN\to \RR$ as above.  The Poisson bracket follows
\begin{equation}\label{rbracket}
\begin{array}{rcl}
\{f_{X_1^*}, f_{X_2^*}\}(p,Y) & = & \la \ad^{\tau}(\Ad(p^{-1})X_1)(Y), \Ad(p^{-1})X_2 \ra \\
& & \qquad  - \la \ad^{\tau}(\Ad(p^{-1})X_2)(Y), \Ad(p^{-1})X_1 \ra \\
& & \qquad  +  \la Y, [\Ad(p^{-1})X_2,  \Ad(p^{-1})X_1] \ra \\ 
& = & \la Y, [\Ad(p^{-1})X_1, \Ad(p^{-1})X_2]\ra - \la Y, [\Ad(p^{-1})X_2, \Ad(p^{-1})X_1]\ra + \\
& & + \la Y, [\Ad(p^{-1})X_2, \Ad(p^{-1})X_1]\ra \\
& = & \la Y, [\Ad(p^{-1})X_1, \Ad(p^{-1})X_2]\ra \\
& = & \la Y, \Ad(p^{-1})[X_1,X_2] \ra \\
& = & f_{[X_1, X_2]^*}(p, Y).
\end{array}
\end{equation}

Note that $f_{X_i^*}(p,Y)=f_{X_j^*}(p,Y)$ for all $(p,Y) \in TN$ if and only if $\la \Ad(p^{-1})X_i, Y\ra=\la \Ad(p^{-1})X_j, Y\ra$ for all $(p,Y)$, and this easily implies  $X_i=X_j$. 

 In the next paragraphs we complete the study of  Killing vector fields corresponding to isometries fixing the identity element for a $k$-step nilpotent Lie group $N$, k=2,3.

A Killing vector field on $N$ for the skew-symmetric derivation $D\in \hh$ at $p=\exp W$,  is given by 
$$D^*(p)=\frac{d}{dt} |_{0} \exp(e^{tD}W) = d (\exp)_W(DW).$$
In fact note that $e^{tD}$ is an orthogonal automorphism of $(\nn, \la\,,\,\ra)$, which  induces a curve of orthogonal automorphisms of $N$, $f_t$,  satisfying the condition  $(d f_t)_e= e^{tD}$.  Also for any automorphism $f:N \to N$, it holds $f(\exp W)=\exp( df_e W)$, that implies the formula above. Finally, the formula for the differential of the exponential map (see for instance Ch. II in \cite{He}) follows:  

$ d \exp_W =  d (L_{\exp W})_e \circ \frac{1 - e^{-\ad(W)}}{\ad(W)} )=  d (L_{\exp W})_e \sum_{i=0} \frac{\ad(-W)^{i}}{(i+1)!}\qquad W\in \nn, $

where $e^A$ denotes  the usual exponential map of linear transformation of $\RR^k$. 
In the case of  a k-step nilpotent Lie group this expression translates into a finite polynomial since $\ad(W)^k=0$ for all $W\in \nn$. Fixing a skew-symmetric derivation on $\nn$, $D\in \hh$,    the  first integral $f_{D^*}:TN \to \RR$ is obtained by the formula:
$$f_{D^*}(\exp(W), Y) = \la D^*(\exp(W)),Y\ra = \la \sum_{i=0}^k \frac{\ad(-W)^{i}}{(i+1)!}(DW), Y\ra.$$
 Explicitly, for the cases 
\begin{enumerate}[(i)]
\item  2-step nilpotent: $f_{D^*}(\exp(W), Y) = \la DW -\frac12 [W, DW], Y\ra$, 
\item  3-step nilpotent: $f_{D^*}(\exp(W),Y) = \la DW -\frac12 [W, DW] + \frac16 [W,[W,DW]], Y\ra$.
\end{enumerate}
One calculates the gradient vector field on $TN$ of the map  $f_{D^*}$, denoted $\grad f_{D^*}(p,Y)=(U,V)$ by
$$U=\frac{d}{dt}|_{0} f_{D^*}(\exp W \exp(sU), Y), \qquad V= d(\exp)_W(DW)=\sum_{i=0} \frac{\ad(-W)^{i}}{(i+1)!}(DW), $$
and one gets that $V=D^*(\exp W)$, while this is completed with:
\begin{enumerate}[(i)]
\item for 2-step: $\grad f_{D^*}(\exp(W), Y) = (-DY + \ad^{\tau}(DW)(Y), DW -\frac12 [W, DW])$, 
\item for 3-step: $\grad f_{D^*}(\exp(W),Y) = (U, V)$ with
 
\qquad $\begin{array}{rcl}
U & = & -DY + \ad^{\tau}(DW)(Y)+ \frac12 \ad^{\tau}([DW,W])(Y) \mbox{and }\\
V & = & DW -\frac12 [W, DW] + \frac16 [W,[W,DW]],
\end{array}
$
\end{enumerate}
where $\ad^{\tau}(X)$ denotes the transpose of $\ad(X)$ with respect to the inner product on $\nn$. 

One proves the next result.

\begin{thm}\label{t2} Let $(N, \la\,,\,\ra)$ denote a $k$-step nilpotent Lie group equipped with a left-invariant metric, and $k=2,3$. The map
$$ \Psi:(\iso(\nn), [\,,\,]) \quad \mapsto \quad  (C^{\infty}(TN),\{\,,\,\}),\, \, \Psi: D + X  \mapsto f_{D^*} + f_{X^*}$$
is a monomorphism from $\iso(\nn)$ on $C^{\infty}(TN)$. 

\end{thm}
\begin{proof} We first prove that $\Psi$ is a Lie algebra homomorphism. Let $D_1, D_2$ denote skew-symmetric derivations on $\nn$, and let $X_1, X_2 \in \nn.$ The Lie bracket on $\iso(\nn)$ is given by 

$$[(D_1,X_1), (D_2,X_2)]=([D_1,D_2], D_1X_2-D_2X_1 + [X_1,X_2]), \, \mbox{for } D_i\in \hh, \, X_i\in \nn,\, i=1,2.$$ 

We shall show the proofs for the 3-step nilpotent case (for $k=2$ the computations are easier). 
The Poisson bracket of the functions $f_{D_1^*}, f_{D_2^*}$ gives $\{f_{D_1^*}$, $f_{D_2^*}\}=f_{[D_1,D_2]^*}$. In fact  by  making use of the gradient vector fields  above, for  $k=3$, one has
$$ 
\begin{array}{rcl}
\{f_{D_1^*}, f_{D_2^*}\} & = & \la -D_1 Y + \ad^{\tau}(D_1W)(Y) +\frac12 \ad^{\tau}([D_1W,W])(Y), D_2 W +\frac12 [D_2 W,W] + \\
 & & + \frac16 [W,[W,D_2W]] \ra  + \\ 
& & - \la -D_2 Y + \ad^{\tau}(D_2W)(Y) + 
 \frac12 \ad^{\tau}([D_2W,W])(Y), D_1 W +\frac12 [D_1 W,W] + \\
& & + \frac16 [W,[W,D_1W]] \ra - \la Y, [D_1 W, +\frac12 [D_1W,W], D_2 W, +\frac12 [D_2W,W]]\ra \\
& = & \la Y, [D_1,D_2](W] +\frac12[[D_1,D_2](W),W] +\frac16 [W,[W, [D_1,D_2](W)]]\ra \\
& = & f_{[D_1,D_2]^*}(p,Y).
\end{array}
$$

Let $D\in \hh, X\in\nn$. In the case $k=3$,  for $p=\exp(W)$ one has
$$
\begin{array}{rcl}
\{f_D, f_{X^*}\} (\exp(W),Y) & = & \la -DY + \ad^{\tau}(DW)(Y)+\frac12\ad^{\tau}([DW,W])(Y), \Ad(p^{-1})X\ra \\
 & & - \la \ad^{\tau}(\Ad(p^{-1})X)(Y), DW + \frac12 [DW,W]+\frac16 [W,[W,DW]]\ra \\
& & + \la Y, [\Ad(p^{-1})X, DW + \frac12[DW,W]]\ra\\
& = & \la Y, D \Ad(p^{-1})X + [DW, \Ad(p^{-1})X] + \frac12 [[DW,W],\Ad(p^{-1})X]\ra \\
& & - \la Y, [\Ad(p^{-1})X,DW] +\frac12[\Ad(p^{-1})X, [DW,W]]\ra + \\
& & + \la Y,[\Ad(p^{-1})X, DW] +\frac12 [\Ad(p^{-1})X, [DW,W]]\\
& = & \la Y, D\Ad(p^{-1})X +[DW, \Ad(p^{-1})X]+\frac12[[DW,W],\Ad(p^{-1})X]\ra.
\end{array}
$$
By using that $\Ad(p^{-1})X= X - [W,X] + \frac12 [W,[W,X]]$, one has

$
D\Ad(p^{-1})X +[DW, \Ad(p^{-1})X]+\frac12[[DW,W],\Ad(p^{-1})X]  \, = $

\qquad  $ = DX - [DW,X] - [W,DX] + \frac12 [DW,[W,X]]  + \frac12 [W,[DW,X]] + \frac12 [W,[W,DX]] + $

\qquad \quad  $+ [DW, X - [W,X]] + \frac12 [[DW,W], X]$ 

\qquad $ = DX - [W,DX] + \frac12 [DW,[W,X]]  + \frac12 [W,[DW,X]] + \frac12 [W,[W,DX]] - [DW, [W,X]] + $

\qquad \quad  $ + \frac12 [[DW,W], X]$

\qquad $ = DX - [W,DX]  + \frac12 [W,[W,DX]] = \Ad(\exp(-W))DX$. 

Since 
$ \frac12 [W,[DW,X]] - \frac12 [DW,[W,X]] - \frac12 [X,[DW,W]] =0$ by the Jacobi identity, one proved the next equality 
$$\{f_{D^*},f_{X^*}\} = f_{(DX)^*}.$$ 
And for $X_1, X_2\in \nn$ it was already  proved above that 
$f_{[X_1,X_2]^*}  \equiv \{f_{X_1^*}, f_{X_2^*}\}$,  which finally proves that $\Psi$ is a Lie algebra homomorphism. 

To prove that $\Psi$ is injective assume now that $f_{D^*}+ f_{X^*}=0$. 

 In particular for $p=e=\exp(0)$ one gets
$0=f_{D^*} (e,Y) + f_{X^*}(e,Y) =\la X, Y\ra=0$ for all $Y\in \nn$, which gives $X=0$. 

So $f_{D^*}\equiv 0$  says that $\la D^*(p), Y\ra =0$ for all $Y\in \nn$. This gives $D^*(p)=0$. From the formulas above one has that
$D^*(p)= d\exp_W (DW)=0$ and since the exponential map is a diffeomorphism on a nilpotent Lie group, its differential is non singular at every point, and this implies  $DW=0$ for all $W\in \nn$, that is $D\equiv 0$.

\end{proof}

From the theorem above it is clear that  abelian subalgebras in the isometry Lie algebra is an important source of first integrals, which could play a role in the question of integrability. Moreover if the dimension of this abelian Lie algebra is maximal. 

\smallskip

\begin{rem} 
In \cite{BT} the authors prove a similar result for pseudo H-type Lie groups, which is a family of 2-step nilpotent Lie groups. They show explicit formulas, using classical tools for studying the geometry of  2-step nilpotent Lie groups. 	
	\end{rem}

 \begin{rem} Assume $(\nn, \la\,,\, \ra)$ is a Lie algebra equipped with a metric, whose  commutator is denoted by $C(\nn)$ and its center  by $\zz$. It is easy to verify that whenever $D:\nn \to  \nn$ is a skew-symmetric derivation, both the commutator and the center are invariant by $D$. Moreover if $\nn =  \vv \oplus \zz$ is a orthogonal decomposition, then $D\vv\subset \vv$.
 Analogously if one has  $\nn =  \ww \oplus  C(\nn)$, a direct sum as vector spaces for $\ww=C(\nn)^{\perp}$. 
\end{rem}

\begin{example} \label{n1} Let $n_1$ denote the 3-step nilpotent Lie algebra of dimension five, with orthonormal basis $e_1, e_2, e_3, e_4, e_5$ and nontrivial Lie brackets:
$$[e_1, e_2]=e_3, \quad [e_1, e_3]=e_5, \quad [e_2, e_4]=e_5.$$
Usual computations show that a  skew-symmetric derivation $D:n_1 \to n_1$ has the form
$$D e_1 = \alpha e_2, \quad De_2 = -\alpha e_1 + \beta e_4, \quad De_4= -\beta e_2 \quad \mbox{ for } \alpha, \beta\in \RR.
$$
\end{example}

\begin{rem}
	
	We conjecture that  Theorem \ref{t2} may still be true for any $k$-step nilpotent Lie group, with $k\geq 2$: there exits a  monomorphism from the isometry Lie algebra to the Lie algebra of first integrals. However the proof may need other tools or techniques.
	
\end{rem}

\section{Invariant First integrals and involution conditions}

In this section we study deeper some involution conditions. Firstly we put attention to invariant first integrals. Next we extend the study to the involution conditions with the first integrals of the previous section. But some results here are true for any Lie group equipped with a left-invariant metric. We need  the nilpotency property  when we play with a first integrals $f_{D^*}$, for a skew-symmetric derivation $D$.

Indeed the Lie group $N$ acts on $TN$ by translations on the left $g \cdot (p,Y) =( gp, Y)$. 
  We say that a function $f:TN \to \RR$ is {\em invariant} if $f(p,Y)=f(e,Y)$ for all $p\in N, Y\in \nn$.

\smallskip

\begin{example}
The energy function on $TN$  is invariant whenever the metric is left-invariant.
\end{example}

 Concerning invariant functions in $C^{\infty}(TN)$ for a Lie group $N$,  we can see the following facts.
\begin{enumerate}[(i)]
\item The gradient  vector field and the Hamiltonian vector  field of an invariant function $f:TN \to  \RR$ are respectively given by
$$\grad \,f(p,Y)= (0,V) \quad X_f(p,Y)=(V, \ad^{\tau}(V)Y)\, \mbox{ where }  V\in \nn, $$
 satisfies $\la V', V\ra = \frac{d}{ds}|_{s=0} f(e, Y +sV')$ for all $V'\in \nn$, and  $\ad^{\tau}(V)$ denotes the transpose of $\ad(V)$ relative to the metric on $\nn$.

The existence of an invariant function $f:TN \to \RR$ is equivalent to the existence of a function on the Lie algebra $\nn$, that is, $F: \nn \to \RR$. 
In fact, given an invariant function $f:TN \to \RR$ define $F: \nn \to \RR$ as
$F(Y)=f(e,Y)$
and conversely given $F:\nn\to \RR$ define an invariant function $f:TN  \to \RR$ by
$$f(p,Y)=F(Y)\qquad \quad \mbox{ for all } p\in N, Y\in  \nn.$$
\item  Let $f_1, f_2: TN \to \RR$ be invariant functions with corresponding gradient vector fields $\grad f_i(p,Y)=(0,V_{f_i})$ for $i=1,2$. Then their Poisson bracket is
\begin{equation}\label{bracket-invariants}
\{f_1,f_2\}(p,Y)= - \la Y, [V_{f_1}, V_{f_2}]\ra.
\end{equation}

Notice that $V_{f_i}=\grad_{\nn}F_i$. In fact let $F_i:\nn \to \nn$ be a smooth function and let $\la\,,\,\ra$ denote the inner product on $\nn$. Then for $U,V\in \nn$ one has 

$\la\grad_{\nn} F_i(Y), V\ra=d{F_i}_Y(V)=d{f_i}_{(p,Y)}(U,V)= d{f_i}_{(p,Y)}(0,V)=\la \grad f_i (p,Y),(0,V)\ra$. 
\end{enumerate}

\begin{lem} \label{lema2} Let $N$ denote a Lie group equipped with a left-invariant metric and let $X^*$ denote a right-invariant vector field $X^*$ with corresponding first integral $f_{X^*}:TN \to \RR$. Then any invariant function  $g:TN \to \RR$ Poisson commutes with any first integral $f_{X^*}$:
$$\{f_{X^*}, g\}\equiv 0 \qquad \mbox{ for all } X\in \nn.$$
\end{lem}

\begin{proof} The proof follows from the information we have. The gradient of $f_{X^*}$ is given by the vector $(\ad^{\tau}(\Ad(p^{-1})X)(Y), \Ad(p^{-1})X)$ and the gradient field  of $g$ has the form $(0,V)$. The formula for the Poisson bracket gives

\smallskip
	
	$\{f_{X^*}, g\}(p,Y)= \la \ad^{\tau}(\Ad(p^{-1})X) Y, V\ra -  \la Y, [Ad(p^{-1})(X), V]\ra=0$, 
	
	\smallskip
	
	which  proves the result. 
\end{proof}
\begin{example} Some interesting functions could be quadratic polynomials on $\nn$ which are obtained as $g(Y)=\la SY,Y\ra$, where $S$ is a symmetric map. Any symmetric map  which gives rise to a invariant first integral  may satisfy the equation  (*) $\la Y, [SY,Y]\ra=0$ for all $Y\in \nn$ -see the next proposition. Any symmetric map $S:\nn \to \nn$ defines  an invariant symmetric Killing 2-tensor in \cite{BM}. 
	
	On the Lie algebra $n_1$ of Example \ref{n1}, any such symmetric map  satisfying (*) may be,  
$$Se_i= a e_i, \mbox{ for } i=1,2,3,4, \quad Se_5 = be_5, \quad 
\mbox{ where } a,b\in \RR.
$$
Note that the identity map corresponds to the symmetric map for  the energy function. So it would be  interesting to find other non trivial examples.
\end{example}
The next proposition specifies conditions for invariant  linear or quadratic polynomials to become first integrals.
\begin{prop}\label{propcuad}
	Let  $(N, \la\,,\,\ra)$ be a Lie group endowed with a left-invariant metric and let $TN$ denote its tangent bundle with the product metric.
	\begin{enumerate}[(i)]
		\item Let $f_{X}: TN \to \mathbb R$ be defined by 
		$$f_{X}(p, Y) = \langle Y, X\rangle.$$
		Then $f_{X}$ is a first integral of the geodesic flow if and only if $\ad(X)$ is skew-symmetric. 
		In particular, for a nilpotent Lie algebra $\nn$, the function  $f_X$ is a first integral if and only if $X\in \zz$. Moreover the family  $\{f_{Z_i}:Z_i\in \zz\}$ is a commutative family of first integrals.
		\item Let $S : \mathfrak n \to \mathfrak n$ be a symmetric endomorphism  and let $g_S:TN \to \RR$ be given by 
		$$
		g_S(p, Y) = \frac{1}{2}\langle Y, SY\rangle.
		$$
		Then $g_S$ is a  first integral of the geodesic flow if and only if
	\begin{equation}\label{eqcuad}
		0 = \langle Y, [SY, Y]\rangle.
	\end{equation}
		
	\end{enumerate}
\end{prop}
\begin{proof} \begin{enumerate}[(i)]
\item It is easy to see that $\grad_{(p, Y)}f_{X} = (0, X)$.  It follows  that $f_{X}$ is a first integral of the geodesic flow if and only if 
$$\la Y, [X,Y]\ra=0 \qquad \mbox{ for all } Y\in \nn,$$
which is equivalent to $\la Y, [X,W]\ra +\la W, [X,Y]=0$ for all $Y,W\in \nn$. The situation for a nilpotent Lie algebra $\nn$ follows from the skew-symmetric property, which implies that $\ad(X)\equiv 0$, that is $X\in \zz$. The last sentence follows from usual computations. 

\item Let $S: \mathfrak n \to \mathfrak n$ denote a symmetric endomorphism of $\mathfrak n$ and let define
$$
g_S(p, Y) = \frac{1}{2}\langle Y, SY\rangle.
$$

An elementary calculation gives 
$$
dg_S|_{(p, Y)}(U, V) = \langle SY, V\rangle
$$ 
and hence 
$$
\grad_{(p, Y)}(g_S) = (0, SY).
$$ 
Then by Equation (\ref{first-int}), we have that $g_S$ is a first integral of the geodesic flow if and only if Equation \eqref{eqcuad} holds.
 
\end{enumerate}
\end{proof}
	
	\begin{example} \label{n23} Let $\nn_{2,3}$ denote the free 3-step nilpotent Lie algebra in two generators. It is a Lie algebra of dimension five with basis $\{e_1,e_2,e_3,e_4,e_5\}$ satisfying the nontrivial Lie bracket relations
$$[e_1, e_2]= e_3 \qquad [e_1, e_3] = e_4, \qquad [e_2, e_3] = e_5.
$$
	Equipp $\nn_{2,3}$ with the metric such that $e_1, e_2, e_3, e_4, e_5$ is orthonormal. One verifies the following assertions by usual computations:
	
	\begin{enumerate}[(i)]
\item A skew-symmetric derivation $D:\nn_{2,3} \to \nn_{2,3}$ is of the form
$$D e_1 = \alpha e_2, \quad De_2 = -\alpha e_1, \quad De_4= \alpha e_5, \quad De_5 = -\alpha e_4, \quad \mbox{ for } \alpha\in \RR.
$$

\item Any symmetric map $S:\nn_{2,3} \to \nn_{2,3}$ giving rise to a first integral $g_S$ as in Proposition \ref{propcuad} has a matrix 
$$\left( \begin{matrix}
a & 0 & 0 & 0  & b-a\\
0 & a & 0 & a-b & 0 \\
0 & 0 & b & 0 & 0 \\
0 & a-b & 0 & c & d \\
b-a & 0 & 0 & d & f
\end{matrix}
\right)
\qquad \quad \mbox{ where } a,b,c,d,e,f\in \RR,
$$
matrix in the basis $e_1, e_2, e_3, e_4, e_5$. 
\end{enumerate}
	\end{example}

	Set a Lie group equipped with a left-invariant metric $(N,\la\,,\,\ra)$, endow its Lie algebra  $\nn$  with the induced metric and denote by  $TN$ its tangent space, trivialized as $TN\simeq N \times \nn$. 
	
	In the next paragraphs explicit involution formulas for first integrals on $TN$ are obtained. 
	
\begin{enumerate}
\item 	Given  elements $U,V\in \nn$ that induce the linear first integrals $f_U, f_V$ with the notation of Proposition \ref{propcuad},   one has
	
	\smallskip
	
	$\{f_U,f_V\}(p,Y)=0 \Longleftrightarrow 0=\la Y, [U,V]\ra \, \forall Y\in \nn \Longleftrightarrow [U,V]=0.$
	
	\smallskip
	
	\item Let $S, T:\nn \to \nn$ be symmetric maps giving rise to quadratic first integrals $g_S, g_T: TN \to \RR$. Then 
	
	\smallskip
	
$\{g_S,g_T\}(p,Y)=0 \Longleftrightarrow 0=\la Y, [SY,TY]\ra \, \forall Y\in \nn.$

\smallskip

\item Consider first integrals $f_U:TN\to \RR$ for $U\in \nn$ and $g_S:TN \to \RR$ from a symmetric map $S:\nn \to \nn$. Then 

\smallskip

$\{f_U,g_S\}(p,Y)=0 \Longleftrightarrow 0=\la Y, [U,SY]\ra \Longleftrightarrow$ $\ad(U)\circ S$ is skew-symmetric on $\nn$. 

For the last statement take $Y=V_1+V_2$ and apply to Equation $0=\la Y, [U,SY]\ra$. 
\end{enumerate}

		The proof of the next result can be done by usual computations and making use of the information given above. 
		
	\begin{lem} \label{involutions} Let $\nn$ denote a Lie algebra equipped with a metric $\la\,,\,\ra$. Consider the corresponding left-invariant metric on its Lie group $N$. The following statements hold:
	\begin{enumerate}[(i)]
	\item Let $U,V\in \nn$ be elements inducing the respective linear first integrals $f_U(p,Y)=\la U, Y\ra$, $f_V(p,Y)=\la V,Y\ra$. Then
	\begin{equation}\label{line}\{f_U,f_V\}(p,Y)=0 \qquad \mbox{ if and only if } \qquad [U,V]=0.
	\end{equation}
	
	\item Let $g_S:TN \to \RR$ denote a first integral induced by a symmetric map $S:\nn \to \nn$. Then 
	\begin{equation}\label{Slin}
	\{f_U,g_S\}(p,Y)=0\quad \mbox{ if and only if } \quad  \ad(U) \circ S \quad \mbox{ is skew-symmetric on $\nn$}.
	\end{equation}
			\item Assume $\nn$ is k-step nilpotent with $k=2,3$. Let $D$ be a skew-symmetric derivation of $\nn$ inducing the first integral $f_{D^*}$. Let $U\in \nn$, which gives the first integral $f_U:TN \to \RR$  and $S:\nn \to \nn$ which gives the first integral $g_S:TN \to \RR$.
		\begin{enumerate}
		\item $\{f_{D^*}, f_U\}=0$ if and only if $DU=0$. 
		\item $\{f_{D^*}, g_S\}=0$ if and only if  the linear map $DS:\nn \to \nn$ is  skew-symmetric.
		\end{enumerate}
	\end{enumerate}
	
	\end{lem}
	
	\begin{proof}
	The proof of (i) and (ii) was given above.

(iii)	Assume now $D:\nn \to \nn$ denote a skew-symmetric derivation. Recall that one has the first integral $f_{D^*}(p,Y)=\la D^*(p), Y\ra$. Making use of the formula of the gradient given in the previous section for the case  3-step nilpotent, one gets
	$$
	\begin{array}{rcl}
	\{f_{D^*}, f_U\} (\exp W, Y) &  =  & \la -DY + \ad^{\tau}(DW)(Y)+ \frac12 \ad^{\tau}([DW,W])(Y), U\ra + \\
	 & & - \la Y, [DW,U] + \frac12 [[DW,W],U] \ra = \la Y, DU\ra.
	\end{array}
	$$
	Thus $\{f_{D^*}, f_U\} (\exp W, Y) = 0$ if and only if $DU=0$. 
	
	By replacing $U$ with $SY$ and following the same computation one gets
	
	\smallskip 
	$\{f_{D^*}, g_S\} (\exp W, Y) = \la Y, DS Y\ra$, 
	
	\smallskip
	
	which says that 
	$\{f_{D^*}, g_S\}\equiv 0$ if and only if $DS$ is  a skew-symmetric map on $\nn$, that is 
	
	\smallskip
	
	$\la Y_1, DS Y_2\ra + \la Y_2, DS Y_1\ra =0$ for all $Y_1, Y_2\in \nn$.
	
	\smallskip
	
	The proofs for the 2-step nilpotent situation can be obtained by easier computations. 
	 
	\end{proof} 
	
\begin{example}\label{rem-heis}
Take the Heisenberg Lie algebra of dimension $2n+1$, denoted $\hh_{2n+1}$ and introduced in Example \ref{heis}. In  \cite{KOR} the geodesic flow was  studied for the canonical metric defined by the orthonormal basis $X_1, \hdots , X_n, Y_1, \hdots, Y_n, Z$. 
Clearly any  symmetric map $A:\vv \to \vv$ can be trivially extended to a symmetric maps on $\nn$.
 The following result was proved (see Theorem 3.2 in \cite{KOR}):

\smallskip

{\em There exists a  bijection, denoted by $\Psi$, between the set of quadratic first integrals of the geodesic flow and the Lie subalgebra of skew-symmetric derivations of $\hh_{2n+1}$. Moreover,}
 $$\{g_{A_1}, g_{A_2}\} =0\quad \mbox{ if and only if }\quad [\Psi A_1, \Psi A_2] =0 \mbox{ in }\sso(\vv).$$
\end{example}

\subsection{Invariant quadratic first integrals on $k$-step nilpotent Lie groups, k=2,3} 
A 2-step nilpotent Lie group corresponds to a 2-step nilpotent Lie algebra $\nn$: the Lie bracket satisfies the relation $[u,[v,w]]=0$ for all $u,v,w\in \nn$. 
There is a natural decomposition of the Lie algebra $\nn$ whenever it is equipped with a metric $\la\, ,\, \ra$:
\begin{equation}\label{split2step}
\nn = \vv \oplus \zz\qquad \mbox{ with } \quad \vv=\zz^{\perp},
\end{equation}
where $\zz$ denotes the center of $\nn$. Notice that the commutator of $\nn$ denoted as $C(\nn)$ is contained in the center, $\zz$. 

Eberlein (Remark 1 in page 472 \cite{Eb2}) and later Del Barco and Moroianu in \cite{BM} described quadratic first integrals on 2-step Lie algebras, in terms of the orthogonal splitting in Equation \eqref{split2step}.    This specifies the conditions given in Proposition \ref{propcuad} for 2-step nilpotent Lie algebras. 

\begin{prop}\label{cuad-2} \cite{BM} Let $(\nn,\la\,,\,\ra)$ denote a 2-step nilpotent Lie algebra equipped with a metric. A symmetric map $S:\nn \to \nn$ gives rise to a first integral if and only if the next conditions occur
\begin{enumerate}[(i)]
\item $[S(X), Y] = [X, S(Y)]$, for all $X, Y\in \vv$,
\item $\ad(X) \circ S|_{\zz}$ is skew-symmetric on $\zz$ for all $X\in \vv$. 
\end{enumerate}
\end{prop}

\begin{proof} Let $S:\nn \to \nn$ be a  symmetric map that induces a first integral, it holds $\la [Y,SY],Y\ra=0$ for all $Y\in \nn$. Write $Y=V + Z$ with respect to the splitting in Equation \eqref{split2step}. Thus 

$$0=\la [V +Z, S(V+Z)], V + Z\ra = \la [V, S(V+Z)], Z\ra = \la [V, S(V)],  Z\ra + \la [V, S(Z)],  Z\ra.$$
From the last equality we may have
$$\la [V, S(V)],  Z\ra = 0 =  \la [V, S(Z)],  Z\ra \quad \mbox{ for all }V\in \vv, Z\in \zz,$$
since the left side is linear in $Z$ but the right one is quadratic. Take $V=X + Y \in \vv$, this gives
$$0 = \la [X, S(Y)],  Z\ra + \la [Y, S(X)],  Z\ra \quad \mbox{ for all } X, Y\in \vv,\, Z \in \zz, $$
and one proves the first assertion. 

For the second assertion take $Z=Z_1+Z_2\in \zz$ and by using the relation $0 =  \la [X, S(Z)],  Z\ra$ one gets
$$0= \la [X, S(Z_1)],  Z_2\ra + \la [X, S(Z_2)],  Z_1\ra \qquad \mbox{ for all }Z_1, Z_2\in \zz, V\in \vv.$$

\end{proof}

For 2-step nilpotent Lie groups, the next invariant first integrals were proposed by Butler in \cite{Bu1}. Assume the orthogonal decomposition of the 2-step nilpotent Lie algebra $\nn$, 
$\nn=\vv \oplus \zz$,  as in Relation \eqref{split2step}, where $\vv=\zz^{\perp}$, and $\zz$ denotes the center of the Lie algebra. A vector $Y\in \nn$ admits a unique  decomposition of the form $Y=V + Z\in \vv \oplus \zz$. 

For $i\geq 0$, define the functions  $g_i:TN \to \RR$ by
\begin{equation}\label{gi}
g_i(p,Y)=\la V, j(Z)^{2i}V\ra\qquad \mbox{ for } Y=V +Z \in \vv\oplus \zz,
\end{equation}
where  $j(Z):\vv \to \vv$ denotes the skew-symmetric map given by $\la j(Z) V,W\ra=\la Z, [V,W]\ra$, for $Z\in \zz$, and $V,W\in \vv$. 

Careful computations show that the gradient of $g_i$ is given by
\begin{equation}\label{gradgi}
\grad g_i(p,Y)=(0, \sum_{j=0}^{2i-1} (-1)^{j+1} [j(Z)^jV, j(Z)^{2i-1-j}V] + 2 j(Z)^{2i}V).
\end{equation}

Therefore it is easy to verify that $g_i$ is a first integral, since

$\la Y, 2 [V,  j(Z)^{2i} V]\ra= \la Z, 2 [V,  j(Z)^{2i} V]\ra = 2\la j(Z) V, j(Z)^{2i} V \ra = 0$, since  $j(Z)^{2i+1}$ is skew-symmetric for all $i\geq 1$. Moreover these functions are pairwise in involution $$\{g_i, g_k\}\equiv 0, $$ 
which follows from the fact that $j(Z)^{2i-2k+1}$ is skew-symmetric if one assumes that $i\geq k$. 

\begin{itemize}
\item Let $g_S$ denote a first integral of the form $g_S(p,Y)=\la SY, Y\ra$ for $S:\nn \to \nn$ being a symmetric map. Then $\{g_i,g_S\}(p,Y)=0$ if and only if 
$$0=\la Z, [j(Z)^{2i}V, S(Z+V)]\ra= \la Z, [j(Z)^{2i}V, S(Z)]\ra +  \la Z, [j(Z)^{2i}V, S(V)]\ra,$$
where one can see  in the last equality that the first term is quadratic on $Z$ while the second one is linear. This implies that
$$\la j(Z)^{2i+1}V, S(Z)\ra = 0 \quad \mbox{and }\quad [S_{|_{\vv}}, j(Z)^{2i+1}]=0.$$
The first equation follows from the fact that $\la Z, [j(Z)^{2i}V, S(Z)]\ra = 0$ -note that $S(Z)$ could have a non trivial component in $\vv$. The last equation is derived from the condition $\la Z, [j(Z)^{2i}V, S(V)]\ra=0$. 

\item Assume $D:\nn \to \nn$ is a skew-symmetric derivation. Let $f_{D^*}$ denote the first integral, $f_{D^*}(\exp W,Y)=\la DW -\frac12 [W,DW],Y\ra$. Then $\{f_{D^*}, g_i\}\equiv 0$ if and only if

\begin{enumerate}[(i)]
	\item $[D_{|_{\vv}},j(Z)^{2i}]=0$ for all $Z\in \zz$ and 
	\item $0=\sum_{j=0}^{2i-1} (-1)^{j+1}\la j(Z) j(DZ) j(Z)^{2i-1-j}V, V\ra =0$.
\end{enumerate}

\end{itemize}

The proof follows by making use of the expressions of the corresponding gradient vector fields. This is summarized in the next result. 

\begin{prop} With the previous notations the invariant maps $g_i$ from Equation \eqref{gi} satisfy
	\begin{enumerate}[(i)]
					\item $\{g_i,g_k\} \equiv 0$, for all $i,k$.
		\item For any symmetric map $S:\nn \to \nn$, take $g_S$ the corresponding invariant first integral, $g_S(p,Y)=\la SY,Y\ra$. Then $\{g_S,g_i\}\equiv 0$ if and only if
		\begin{enumerate}
			\item $\la j(Z)^{2i+1}V, S(Z)\ra = 0$ for all $Z\in \zz, V\in \vv$ and \quad \item $ [S_{|_{\vv}}, j(Z)^{2i+1}]=0$, where  $S_{|_{\vv}}:\vv \to \vv$. 
		\end{enumerate}
	\item Let $D:\nn \to \nn$ denote a skew-symmetric derivation which gives rise to a first integral $f_{D^*}$. Then $\{f_{D^*},g_i\}\equiv 0$ if and only if
	\begin{enumerate}
		\item $[D_{|_{\vv}}, j(Z)^{2i}]=0$, for $D_{|_{\vv}}:\vv \to \vv$ and
		\item $\sum_{j=0}^{2i-1}(-1)^{j+1} \la j(Z) j(DZ) j(Z)^{2i-1-j}V, V\ra=0$, for all $Z\in \zz, V\in \vv$. 
	\end{enumerate}

	\end{enumerate}
\end{prop}

Notice that a skew-symmetric derivation $D:\nn \to \nn$ satisfies, $D\zz \subseteq \zz$, $D\vv \subseteq \vv$ and $j(DZ)=[D,j(Z)]$, with respect to the orthogonal decomposition $\nn=\vv \oplus \zz$. 

\smallskip

Assume now that $\nn$ is a 3-step nilpotent Lie algebra (thus, $[[[U,V],X],W]=0$ for all $X,U,V,W\in \nn$). 
Let $\la\,,\,\ra$ denote a metric on $\nn$ and denote by $C(\nn)$ the commutator of $\nn$. One has the orthogonal decomposition 
\begin{equation}\label{dec-3n}
\nn = \vv \oplus C(\nn) \qquad \mbox{ where } \vv= C(\nn)^{\perp}.
\end{equation}
In this case it holds $\zz \supseteq C^2(\nn)$ where $C^2(\nn)$ is the ideal spanned  by elements of the form $[[U,V],W]$. And one has a result for symmetric maps $S:\nn \to \nn$ to induce first integrals, similar to the 2-step situation.

\begin{prop} \label{cuad-3}
Let $(\nn,\la\,,\,\ra)$ denote a 3-step nilpotent Lie algebra equipped with a metric. A symmetric map $S:\nn \to \nn$ gives rise to a first integral if and only if the next conditions occur
\begin{enumerate}[(i)]
\item $[SX, Y] = [X, SY]$, for all $X, Y\in \vv$,
\item $(\ad(X) \circ S - \ad(SX))|_{C(\nn)}$ is skew-symmetric on $C(\nn)$ for all $X\in \vv$. 
\item $\la X, [SX,X]\ra=0$ for all $X\in C(\nn)$.
\end{enumerate}
\end{prop}

\begin{proof}  Let $U\in C(\nn)$ and $X\in \vv$. A symmetric map $S:\nn \to \nn$ gives rise to a first integral $g_S$ if and only if 
$$ \begin{array}{rcl}
0 & = & \la [S(X+ U), X+U], X+U\ra = \la [S(X+ U), X+U], U\ra \\
 & = &  \la [S(X), X], U\ra + \la [S(X), U], U\ra + \la  [S(U), X], U\ra.
\end{array}
$$
Notice that while the first term in the last equality is linear on $U$, the others are quadratic, and this gives
$$\la [S(X), X], U\ra = 0 =\la [S(X), U], U\ra + \la  [S(U), X], U\ra.$$
Similar to the proof for case 2-step nilpotent we get that $\la [S(X), X], U\ra = 0$ is equivalent to 
$[S(X), Y] = [X, S(Y)] $ for all $X,Y\in\vv$. 

On the other hand one has  the equality $0 =\la \ad(S(X)) U, U\ra - \la  \ad(X)(S(U)) , U\ra,$ for all $X\in \vv, U\in C(\nn)$,  which with similar techniques already used gives rise to the fact that the map $(\ad(S(X)) -\ad(X) \circ S)_{|_{C(\nn)}}$ is skew-symmetric on $C(\nn)$ for all $X\in \vv$. 

If one assumes that $S$ gives rise to a first integral, then clearly $\la U, [SU,U]\ra=0$ for all  $U\in C(\nn)$. Assuming (i),(ii) and (iii), one gets the condition for $g_S$ to be a first integral.

\end{proof}

\section{applications on low dimensions}

In this section we apply the results of the previous sections for the construction of first integrals on nilpotent Lie algebras of dimensions at most six. 

A list  of nilpotent Lie algebras of dimension at most six was obtained by the de Graaf in  \cite{Gra}, see section 4. From this we get the k-step nilpotent -non abelian- Lie algebras, with $k=2,3$.  According to the dimension one has

\begin{enumerate}[(i)]
\item Dimension three: The Heisenberg Lie algebra $\hh_3$ (see Example \ref{heis}).
\item Dimension four: 
\begin{enumerate}
\item The trivial extension $\RR \oplus \hh_3$ and

\item  The Lie algebra $n_2$ from Example \ref{n2}. 
\end{enumerate}

\item Dimension five. Assume the Lie algebra is spanned by vectors $e_1, e_2, e_3, e_4, e_5$.  
 \begin{enumerate}
\item The trivial extensions $\RR^2 \oplus \hh_3$, and $\RR \oplus n_2$. 
\item The Heisenberg Lie algebra of dimension five, $\hh_5$ (see Example \ref{heis}).
\item The Lie algebra $n_1$ of  Example \ref{n1}. 
\item  The ``star'' Lie algebra $n_3$, with Lie brackets $[e_1, e_2]=e_4$, $[e_1, e_3]= e_5$. 
\item The free 3-step nilpotent Lie algebra in two generators $\nn_{2,3}$ of Example \ref{n23}. 
\end{enumerate}

Note that the Lie algebra $n_3$ can be constructed starting with the  star graph in three vertices, see more details in \cite{Ov1}.

There exist two more nilpotent  Lie algebras in dimension five, which are 4-step nilpotent: 
\begin{enumerate}
\item 4-step almost abelian: $[e_1,e_2]=e_3, \, [e_1, e_3]=e_4, \, [e_1, e_4]=e_5.$
\item 4-step non almost abelian: $[e_1,e_2]=e_3, \, [e_1, e_3]=e_4, \, [e_1, e_4]=e_5,\, [e_2, e_3]=e_5.$
\end{enumerate}
\end{enumerate}
Recall that a Lie algebra $\nn$ is called {\em almost abelian} if it has a codimension one abelian ideal.

A goal now is to prove the next theorem and to study  the situation in dimension six. 

\begin{thm} \label{t3} For any $k$-step nilpotent Lie algebra, k=2,3, of dimension at most five, there exists a left-invariant metric whose geodesic flow is Liouville integrable. 

 \end{thm}

The proof of this theorem will be given in two steps. In the first one, we exhibit a {\em complete set of first integrals}, that is a maximal set of pairwise commuting first integrals. This is contained in Lemma \ref{firstlem} and secondly we complete the proof of the integrability by proving  the linear independence.

\begin{lem}\label{firstlem} For any $k$-step nilpotent Lie algebra, k=2,3, of dimension at most five, there exists a left-invariant metric whose geodesic flow admits a complete set of first integrals. 
	
	Moreover, except for the free 3-step nilpotent Lie algebra in two generators, a complete set of commuting first integrals is constructed from Killing vector fields, - see Theorem \ref{t2}.
	
	\end{lem}

Note that the integrability of the geodesic flow on the  trivial extension of a given Lie group $N$, that is,  $\RR^k \times N$,  depends on  the integrability of the geodesic flow on $N$. In fact in any case, for every central vector $X$ in the abelian factor, one gets  a linear first integral $f_{X}$ which  commutes with any first integral defined on $TN$ (and extended to $T(\RR^k \times N)$) as proved in Lemma \ref{involutions}. This holds in any dimension.

\begin{proof}


\begin{enumerate}[(i)]
\item Dimension $3$. For the Lie algebra $\hh_3$ with basis $X_1, Y_1, Z$ and non trivial relations $[X_1, Y_1]=Z$, take  the metric that makes of this basis an orthonormal basis. 

Then the set $\{f_{Z}, \En, f_{X_1^*}\}$ is a set of pairwise commutative first integrals for the geodesic flow. In fact both $f_{Z}, \En$ are invariant functions each of them commuting with $f_{X_1^*}$.

\item Dimension $4$. The Lie algebra $n_2$ has a codimension one abelian subalgebra. For instance take $\aa=span\{e_2, e_3, e_4\}$. Thus for any left-invariant metric on the corresponding Lie group, the set of first integrals given by $\{\En, f_{e_2^*}, f_{e_3^*}, f_{e_4^*}\}$ builds a complete set of first integrals for the geodesic flow, that is a direct application of Theorem \ref{t2}.

\item Dimension 5. Let $\hh_5$ denote the Heisenberg Lie algebra of dimension five. This is spanned by the vectors $X_1, X_2, Y_1, Y_2, Z$ satisfying the Lie bracket relations $[X_i,Y_j]=\delta_{ij} Z$. 

Let $S_i:\nn \to \nn$ be the symmetric maps given by $S_i:=Id|_{\vv_i}$,  where $\vv_i$ is the subspace of $\hh_5$ spanned as $\vv_i=span\{X_i,Y_i\}$ for $i=1,2$, and define the invariant maps on $TH_5$ by
$g_{S_i}(p,Y)=\la S_i Y,Y\ra$.  Then the set $\{f_Z,f_{X_1^*}, f_{X_2^*}, g_{S_1}, g_{S_2}\}$ is a complete set of first integrals. This follows as application of Theorem \ref{t2} and Lemma \ref{lema2}.

\item Dimension $5$. Let $n_3$ denote the Lie algebra spanned by vectors $e_1,e_2,e_3,e_4,e_5$ satisfying the non trivial Lie brackets $[e_1, e_2]=e_4$, $[e_1, e_3]= e_5$. Let $N$ denote the corresponding Lie group equipped with a left-invariant metric. Since the ideal of $n_3$ spanned by $e_2,e_3,e_4,e_5$ is abelian, the geodesic flow admits a complete set of first integrals, as proved in Theorem \ref{t2}.

\item Dimension $5$. Let $n_1$ denote Lie algebra  of  Example \ref{n1}. Clearly the commutator is spanned by the vectors $e_3, e_5$ while the center by the vector $e_5$. One has an abelian subalgebra spanned by the vector $e_3, e_4, e_5$. 

Fix the left-invariant metric on the Lie group for which the basis $e_1, e_2, e_3, e_4, e_5$ is orthonormal. Let $D:n_1\to n_1$ be the skew-symmetric derivation such that $De_1=e_2, De_2=-e_1$ and $De_i=0$ for all i=3, 4, 5.

An application of Theorem \ref{t2} and Lemma \ref{involutions} shows that the set given by  $\{\En, f_{e_3^*},$ $ f_{e_4^*}, f_{e_5}, f_{D^*}\}$ is a set of pairwise commuting first integrals. 

\item Dimension $5$.  Let $\nn_{2,3}$ be the free 3-step nilpotent Lie algebra in two generators  of Example \ref{n23}. Indeed the center is spanned by the vectors $e_4, e_5$ and the commutator by the vectors  $e_3, e_4, e_5$. 

Take the left-invariant metric on the Lie group for which this a orthonormal basis. Take the symmetric map on $\nn_{2,3}$ given by 
$$Se_1=e_5,\quad Se_2=-e_4, \quad Se_3=e_3, \quad S e_4=-e_2, \quad Se_5=e_1.$$
 This induces a first integral $g_S: TN_{2,3} \to \RR$, in fact $S$ satisfies conditions of Proposition \ref{cuad-3}. Thus the set $\{\En, f_{e_3^*}, f_{e_4}, f_{e_5}, g_S\}$ is a set of pairwise commuting first integral, as it satisfies conditions of Lemma \ref{lema2}, Lemma  \ref{involutions} and Theorem \ref{t2}.

\end{enumerate}
\end{proof}

Next we complete the proof of Liouville integrability. We also show examples of Liouville integrability of the geodesic flow on some  compact quotients, when they are  equipped with the  induced metric, by asking the projection map is a local isometry.

\smallskip

\begin{itemize}
	\item On the Heisenberg Lie group $H_3$ endowed with the canonical metric as in Example \ref{heis}, a complete set of first integrals on $TH_3$ corresponds to the set 
	$\En, f_{X_1^*}, f_Z$, where
	$f_{X_1^*}(p,Y)=\la X_1, Y\ra + \la W, Y_1\ra \la Z, Y\ra$, with $\exp(W)=p\in H_3$.
	
	In fact the gradient vector fields at $p=(x_1,y_1,z)$ are given by
	$$\grad f_Z(p,Y)=(0,Z), \, \grad\En(p,Y)=(0,Y), \, \grad f_{X_1^*}(p,Y)=(\la Z, Y\ra Y_1, X_1+ y_1 Z).$$
	The set $\{f_{Z}, \En, f_{X_1^*}\}$ is linear independent on the dense set of $TH_3$ defined by $\{(p,Y): \la Y,Z\ra \neq 0 \mbox{ and } \la Y,X_1\ra \neq 0 \mbox { or } \la Y, Y_1 \ra \neq 0\}$, which proves that the geodesic flow is Liouville integrable. 
	
	On the Heisenberg Lie group $H_5$ equipped with the canonical metric as in Example \ref{heis}, one has the  first integrals $f_Z, f_{X_1^*}, f_{X_2^*}$ which pairwise commute. Let $S_i:\nn \to \nn$ be the symmetric map given by $S_i:=Id|_{\vv_i}$ where $\vv_i=span\{X_i,Y_i\}$ for $i=1,2$, and define the invariant maps on $TH_5$ by
	$g_{S_i}(p,Y)=\la S_i Y,Y\ra$.  Then the set $\{f_Z,f_{X_1^*}, f_{X_2^*}, g_{S_1}, g_{S_2}\}$ is a set of commuting first integrals which are linear independent (see more details in \cite{KOR}). 
	
	A lattice in $H_3$ or $H_5$ can be constructed as
	$$\Gamma_r= r\ZZ \times \ZZ \times \ZZ, \quad  \Gamma_{r_1,r_2}=r_1\ZZ \times r_2 \ZZ \times \ZZ \times \ZZ \times \ZZ, \, \mbox{ where} \, r, r_1, r_2\in \NN,\, r_1|r_2.$$
	
	Equipp  the quotient $\Gamma_r \backslash H_3$ with the induced metric for which the projection map is a local isometry. Then any invariant function on $TH_3$ is induced to the quotient. The function $f_{X_1^*}$ induces the following differentiable function on $T(\Gamma_r \backslash H_3)$:
	$$f_1(\Gamma_rp,Y)= e^{-1/f_Z^2} \sin \left(2\pi\frac{f_{X_1^*}}{f_Z}\right)(p,Y),$$
	and similarly one construct the functions $f_1,f_2$ on $T(\Gamma_{r_1,r_2}\backslash H_5)$ induced by $f_{X_1^*}, f_{X_2^*}$. This prove the Liouville integrability, the linear independece also  follows from  properties of the Poisson bracket.
	
	\item Dimension 4. 	For the Lie 3-step nilpotent Lie algebra $n_2$, one has the corresponding Lie group modeled on $\RR^4$ with the usual topology and with  the multiplication map  given by
	$$\begin{array}{rcl}
	(x_1,x_2,x_3,x_4)(x_1',x_2',x_3', x_4') & = & (x_1+x_1',x_2+x_2',x_3+x_3'+\frac12(x_1x_2'-x_2x_1'),\\
	& & \quad x_4+x_4'+\frac12(x_1x_3'-x_1'x_3)+\frac{x_1}{12}(x_1x_2'-x_2 x_1')).
	\end{array}$$ 
	It is not hard to see that the next vector fields are  Killing vectors: 
	$$\begin{array}{rcl}
	e_2^*(x_1,x_2,x_3,x_4) & = & e_2 -x_1 e_3 + \frac5{12} x_1^2 e_4,\\
	e_3^*(x_1,x_2,x_3,x_4) & = & e_3 - x_1 e_4,
	\end{array}
	$$
	where $e_1,e_2,e_3,e_4$ denote the basis of orthonormal left-invariant vector fields. 
	
	We have the set of first integrals $\En,f_{e_4}, f_{e_2^*}, f_{e_3^*}$ with corresponding gradients at $(p,Y)$, with $p=(x_1,x_2,x_3,x_4)$:
	
	$\grad \En (p,Y)=(0,Y),\qquad \grad f_{e_2^*}(p,Y)=((-\la Y, e_3\ra + x_1 \la Y,e_4\ra) e_1, e_2^*),$
	
	$\grad f_{e_4^*}(p,Y)=(0,e_4),\qquad \grad f_{e_3^*}(p,Y)=( - \la Y,e_4\ra e_1, e_3^*).$
	
	From this usual computations show that this set is linear independent whenever $\la Y, e_1\ra \neq 0$.

\item The Lie algebra $n_3$ of dimension five is 2-step nilpotent, where the vectors $e_2, e_3, e_4, e_5$ span an  abelian subalgebra. For any left-invariant metric on $TN_3$ one gets  the first integrals $f_{e_i^*}$ which are pairwise in involution.

The corresponding simply connected Lie group can be modeled on $\RR^5$ with the multiplication map given by
$$
(v,x_4,x_5)(w,x_4',x_5')  = (v+w, x_4+x_4'+\frac12(x_1x_2'-x_2x_1'), x_5+x_5'+\frac12(x_1x_3'-x_3x_1')),
$$
where $v=(x_1,x_2,x_3), \, w=(x_1',x_2', x_3')$. One has the right-invariant vector fields at every point $p=(x_1,x_2,x_3,x_4,x_5)\in N_3$:
$$e_2^*(p)=e_2(p)-x_1e_4(p), \qquad e_3^*(p)=e_3(p)-x_1e_5(p).$$
 Let $\la \,,\ra$ denote the metric for which the vectors $e_1,e_2,e_3,e_4,e_5$ build a orthonormal basis; this induces on $N_3$ a left-invariant metric, which on $\RR^5$ is given in canonical coordinates as
 
 \smallskip
 
 $g=(1+\frac14(x_2^2+x_3^2))dx_1^2 +(1+\frac14 x_1^2) (dx_2^2 +dx_3^2) +dx_4^2 +dx_5^2-\frac14 x_1 x_2 dx_1dx_2  $
 
 \smallskip
 
 $\qquad -\frac14 x_1 x_3 dx_1dx_3 +
 \frac12 x_2 dx_1 dx_4 + \frac12 x_3 dx_1 dx_5 -\frac12 x_1 (dx_2 dx_4 +dx_3 dx_5). 
 $
 
 \smallskip
 
  We have the first integrals $\En,f_{e_4}, f_{e_5}, f_{e_2^*}, f_{e_3^*}$ with corresponding gradient vector fields  given by:

$\grad \En(p,Y)=(0,Y), \quad \grad f_{e_4}(p,Y)=(0,e_4),\quad \grad f_{e_5}(p,Y)=(0,e_5),$

$\grad f_{e_2^*}(p,Y)=(\la Y, e_4\ra e_1,e_2^*), \qquad \grad f_{e_3^*}(p,Y)=(\la Y, e_5\ra e_1,e_3^*)$. 
	
	It is not hard to see that if $\la Y,e_1\ra\neq 0$, the set $\En,f_{e_4},f_{e_5},f_{e_2^*},f_{e_3^*}$ is linear independent proving the Liouville integrability of the geodesic flow. 
	
	The Lie group $N_3$ admits at least a cocompact  lattice. In fact, assume $r\in \NN$, $r\geq 2$ and take $\Lambda_r= r\ZZ \times \ZZ \times \ZZ \times \ZZ \times \ZZ \subset N_3$, which is clearly a discrete subgroup such that the quotient $\Lambda_s \backslash N_3$ is a compact space.  Take the metric on $\Lambda_{r}\backslash N_3$ for which the canonical projection $\pi:N_3 \to \Lambda_{r}\backslash N_3 $ is a local isometry. 
	
	The geodesic flow on $T(\Lambda_{r}\backslash N_3)$ is Liouville integrable since all first integrals we had on $N_3$ can be induced to the quotients. In fact, the invariant functions $\En,f_{e_4},f_{e_5}$ are directly induced and take the differentiable functions
	
	$f_2(\Lambda_{r}p,Y)=e^{-1/\la Y,e_4\ra^2} sin\left( 2\pi \frac{f_{e_2^*}}{f_{e_4}}\right)(p,Y)$, 
	
	\smallskip
	
	$f_3(\Lambda_{(r,r_1,r_2,m_1,m_2)}p,Y)=e^{-1/\la Y,e_5\ra^2} sin\left( 2\pi \frac{f_{e_3^*}}{f_{e_5}}\right)(p,Y)$, 
	
	that finishes the proof of the complete integrability in this case. See more details in \cite{Ov1}. 
	
\item Take the 3-step nilpotent Lie algebra $n_1$. The corresponding Lie group can be modeled on $\RR^5$ together with the multiplication map given by
$$
\begin{array}{rcl}
(v,x_3,x_5)(w,x_3',x_5') & = &(v+w, x_3+x_3'+\frac12(x_1x_2'-x_2x_1'), x_5+x_5' +\\
& &  + \frac12(x_1x_3'-x_3x_1'+x_2x_4'-x_4x_2') + \\
& &  +\frac1{12}[x_1(x_1x_2'-x_2x_1')+x_1'(x_1'x_2-x_1x_2')]),
\end{array}
$$	
	where we take the notation $v=(x_1,x_2,x_4), w=(x_1',x_2',x_4')$. By computing one gets the next right-invariant vector fields  at $p=(x_1,x_2,x_3,x_4,x_5)$:
	$$e_3^*(p)=e_3-x_1 e_5, \qquad e_4^*(p)=e_4-x_2e_5,$$ 
	where $e_3,e_4, e_5$ are the left-invariant vector fields, which in usual coordinates  for $p=(x_1,x_2,x_3,x_4,x_5)$ are given by
	$$e_3=\partial_{x_3} +\frac12 x_1\partial_{x_5}, \quad e_4=\partial_{x_4} +\frac12 x_2\partial_{x_5},\qquad  e_5=\partial_{e_5}.$$
		As the metric on $n_1$ is taken such that the basis $e_1, e_2, e_3, e_4, e_5$ is orthonormal, one gets a left-invariant metric on $N_1$.
		
		The skew-symmetric derivation given by $De_1=e_2, De_2=-e_1$ and $De_i=0$ for $i=3,4,5$, gives rise to the Killing vector field at $p=(x_1,x_2,x_3,x_4,x_5)$:
	$$D^*(p)=-x_2 e_1 + x_1 e_2 -\frac12 (x_1^2+x_2^2)e_3+[\frac12(x_1x_4-x_2x_3)+\frac{x_1}6(x_1^2+x_2^2)] e_5,$$
	which induces the first integral defined by $f_{D^*}(p,Y)=\la D^*(p), Y\ra$. This has a gradient vector field $\grad f_{D^*}(p,Y)=(U,V)$, for  $p=(x_1,x_2,x_3,x_4,x_5)$, and  for  $Y=y_1e_1+y_2e_2+y_3e_3+y_4e_4+y_5 e_5$ given by 
	$$\begin{array}{rcl}
	U & = & (y_2-x_1y_3+\frac12(x_1^2+x_2^2)y_5)e_1- (y_1+x_2y_3)e_2-x_2y_5 e_3 + x_1y_5 e_4\\
	V & = & -x_2 e_1 + x_1 e_2-\frac12(x_1^2+x_2^2) e_3 +[\frac12(x_1x_4-x_2x_3)+\frac16 x_1(x_1^2+x_2^2)] e_5.
	\end{array}
	$$
By using this it is not hard to see that the first integrals $\En, f_{e_5}, f_{e_3^*}, f_{e_4^*}, f_{D^*}$ are linear independent if one asks $y_5\neq $ and either $x_1\neq 0$ or $x_2\neq 0$, conditions that shape a dense set on $TN_1$.

\item For the free 3-step nilpotent Lie algebra $\nn_{2,3}$, with orthonormal basis $e_1, e_2, e_3, e_4, e_5$ we model the corresponding Lie group on $\RR^5$ with the multiplication given by
$$\begin{array}{rcl}
(x_1,x_2,x_3,x_4,x_5)(x_1',x_2',x_3',x_4',x_5') & = &(x_1+x_1',x_2+x_2', x_3+x_3'+\frac12(x_1x_2'-x_2x_1'),\\
& & x_4+x_4'+\frac12(x_1x_3'-x_3x_1')+\frac{x_1}{12}(x_1x_2'-x_2x_1'),\\
& &  x_5+x_5'+\frac12(x_2x_3'-x_3x_2')+\frac{x_2}{12}(x_1x_2'-x_2x_1')).
\end{array}
$$	
We have the invariant functions $\En, f_{e_4}, f_{e_5}, g_S$. The set $\{\En, f_{e_3^*}, f_{e_4}, f_{e_5}, g_S\}$ is a set of pairwise commuting first integral, as it satisfies conditions of Lemma \ref{lema2}, Lemma  \ref{involutions} and Theorem \ref{t2}.

For the invariant functions we have the corresponding gradient vector fields:
$$\begin{array}{rcllcl}
\grad \En(p,Y) & = & (0,Y), \qquad  \grad g_S (p,Y) & = &(0,SY)\\
\grad f_{e_4} (p,Y) & = & (0,e_4), \qquad \grad f_{e_5} (p,Y) & = & (0,e_5),
\end{array}
$$
while for the function $f_{e_3^*}(p,Y)=\la e_3-x_1e_4- x_2 e_5,Y\ra$ for $p=(x_1,x_2,x_3,x_4,x_5)$, the corresponding gradient vector field is at $(p,Y)$ with $Y=\sum_{i=1}^5 y_i e_i$,
$$\grad f_{e_3^*}(p,Y)=(-y_4 e_1 - y_5 e_2, e_3 -x_1 e_4 -x_2 e_5).$$

This set  of first integrals is linear independent on the dense set given by $\{(p,Y): \la Y, e_4\ra \neq 0\, \mbox{ and } \la Y, e_2 \ra \la Y, e_5\ra + \la e_1, Y\ra \la Y, e_4 \ra \neq 0 \}$, which proves the Liouville integrability.
\end{itemize}

\subsection{Dimension six} Now we study the situation in dimension six.  With the  notation of de Graaf \cite{Gra} the list of k-step nilpotent Lie algebras of dimension six, with k=2,3, is given by:

 \begin{enumerate}
\item The trivial extensions of low-dimensional Lie algebras. 
\item $n_{6,10}$: $[e_1,e_2]=e_3, \quad [e_1, e_3]=e_6, \quad [e_4, e_5]=e_6$.
\item $n_{6,19 (\varepsilon)}$: $[e_1, e_2] = e_4,\quad [e_1, e_3] = e_5,\quad [e_2, e_4] = e_6, \quad [e_3, e_5] = \varepsilon e_6.$
\item $n_{6,20}$: $[e_1,e_2]=e_4, \quad [e_1, e_3]=e_5, \quad [e_1, e_5]=e_6=[e_2, e_4]$.
\item  $n_{6,22(\varepsilon)}:$ $[e_1, e_2] = e_5,\quad [e_1, e_3] = e_6, \quad [e_2, e_4] = \varepsilon
e_6,\quad [e_3, e_4] = e_5$.
\item $n_{6,23}:$ $[e_1, e_2] = e_3,\quad [e_1, e_3] = e_5,\quad [e_1, e_4] = e_6, \quad [e_2, e_4] =  e_5.$
\item $n_{6,24(\varepsilon)}:$ $[e_1, e_2] = e_3,\quad [e_1, e_3] = e_5, \quad [e_1, e_4] = \varepsilon
e_6,\quad [e_2, e_3] = e_6,\quad [e_2, e_4] = e_5$. 
\item $n_{6,25}:$ $[e_1, e_2] = e_3,\quad [e_1, e_3] = e_5,\quad [e_1, e_4] = e_6.$
\item $n_{6,26}:$ $[e_1, e_2] = e_4,\quad [e_1, e_3] = e_5, [e_2, e_3] = e_6$.
\end{enumerate}

On the corresponding Lie group fix the left-invariant metric, which at the Lie algebra corresponds to the orthonormal basis $e_1, e_2, e_3, e_4, e_5, e_6$. Next we shall construct a complete set of first integrals in involution by applying the previous results, if possible. 

\begin{enumerate}[(i)]

\item $n_{6,10}$: This is a 3-step nilpotent Lie algebra such  that the Lie subalgebra spanned by $e_2, e_3, e_5, e_6$ denoted by $\aa$ is abelian.  The first integrals  $f_{e_i^*}$ for $i=2,3,5,6$ is a set of pairwise commuting first integrals. 

The skew-symmetric derivation $D$ of $n_{6,10}$  defined on the basis $e_1, e_2, e_3, e_4, e_5, e_6$ by 
$$De_1=e_4,\quad De_4=-e_1, \quad De_i=0 \mbox{ for } i\neq 1,4,$$ gives rise to the first integral $f_{D^*}$ that Poisson commutes with the first integrals $f_{e_i^*}$ for $i=2,3,5,6$. The set 
$$\{\En, f_{D^*}, f_{e_i^*}\}\quad \mbox{ for } i=2,3,5,6,$$ is a complete set of first integrals. 

\item $n_{6,19}(\varepsilon)$  In this case any symmetric map satisfying Equation  \eqref{eqcuad} in the orthonormal basis $e_1, e_2, \hdots, e_6$ has the form
$$ \bullet \,  \varepsilon =0 \, \left( 
\begin{matrix}
\alpha & 0  & 0 & 0  & 0 & 0\\
0 & \alpha & 0 & 0 & 0  & 0 \\
0 & 0 & \alpha & 0 & 0 & 0\\
0 & 0 & 0 & \alpha & 0 & 0 \\
0 & 0 & 0 & 0 & s_1 & s_2\\
0 & 0 & 0 & 0 & s_2 & s_3
\end{matrix}
\right), \qquad  \bullet \,  \varepsilon \neq 0 \, \left( 
\begin{matrix}
\alpha & 0  & 0 & 0  & 0 &\varepsilon  \gamma\\
0 & \alpha & 0 & 0 & 0  & 0 \\
0 & 0 & \alpha & 0 & 0 & 0\\
0 & 0 & 0 & \alpha + \varepsilon  \gamma & 0 & 0 \\
0 & 0 & 0 & 0 & \gamma  & 0\\
\varepsilon  \gamma & 0 & 0 & 0 & 0 & \delta
\end{matrix}
\right),
$$
where all variables are assumed to be real numbers. 

For $\varepsilon \neq 0$ the set of first integrals given by
$\{\En, f_{e_i^*}, S\}$ for $i=3,4,5,6$ and $S:=S_1$ the symmetric map with $\gamma=1$ shows a  complete set of  functions in involution. 

For $\varepsilon = 0$, let $D$ denote the skew-symmetric derivation defined by $De_1=e_2$, $De_2=-e_1$ and $De_i=0$ for $i\neq 1,2$. The set of first integrals given by
$\{\En, f_{e_i^*}, f_{D^*}\}$ for $i=3,4,5,6$  shows a maximal set of pairwise commuting first integrals. 

\item $n_{6,20}$.  Let $D$ be the skew-symmetric derivation given  by
$De_1=e_2,$, $ De_2=-e_1$, $De_i=0$ for $i=3,4,5,6$. As application of Theorem \ref{t2} the next set  is a complete set of first integrals. 
$$\{\En, f_{D^*}, f_{e_i^*}\}\quad  i=3,4,5,6.$$

\item $n_{6,22}(\varepsilon)$. This is a 2-step nilpotent Lie algebra, which is almost non-singular. An abelian Lie subalgebra is spanned by the vectors $e_1, e_4, e_5, e_6$ and one has the invariant first integral constructed as in Equation \eqref{gi}:
$$g(p,V+Z)=\la V, j(Z)^2 V\ra, \qquad \mbox{ for all } V+Z \in \vv\oplus \zz, \mbox{ with } \vv=\zz^{\perp}.$$
Note that $\zz=span\{e_5, e_6\}$ and $\vv=span\{e_1, e_2, e_3, e_4\}$. Therefore the map $g$ is a polynomial of degree four explicitly given by
$$g(p,V+Z)=-(z_5^2+z_6^2)[v_1^2+v_4^2+(1+\varepsilon)(v_2^2+v_3^2)] + 2 (1 + \varepsilon) z_5 z_6 [v_1v_4-v_2 v_3].$$

The involution of the set of first integrals $\{\En, g, f_{e_1^*},  f_{e_4^*}, f_{e_5}, f_{e_6}\}$ follows from Theorem \ref{t2} and Lemma \ref{lema2}. 

\item $n_{6,23}$. In this case any skew-symmetric derivation is trivial and any symmetric map on the Lie algebra has the form:
$$S e_i =\alpha e_i, \quad i=1,2,4, \quad Se_3= \beta e_3, \quad Se_5= s_{11} e_5 + s_{12} e_6, S e_6= s_{12} e_5 + s_{22} e_6, 
$$ 
where all constrains are real numbers. 

Denote by $S_1, S_2$ corresponding symmetric maps on the Lie algebra with $\alpha_i, \beta_i$ i=1, 2 (and the block for the center free of conditions). One gets that the first integral $g_{S_1}$ commutes with $g_{S_2}$ if and only if $\alpha_1 \beta_2 -\alpha_2 \beta_1 =0$. But since the energy function belongs to this set, one should find another symmetric map. Writing $\alpha_1=1=\beta_1$ one gets that $\alpha_2=\beta_2$ and in this case, the symmetric map $S_2$ results a multiple of the identity map.

Therefore in this case we cannot have a complete set of  first integrals by applying the methods of the previous sections.

\item $n_{6,24}(\varepsilon)$.  This is a 3-step nilpotent Lie algebra with abelian subalgebra $\aa$ spanned by the vectors $e_3,e_4,e_5,e_6$, which give rise to the commuting set of first integrals $f_{e_i^*}$,$ i=3,4,5,6$. To complete this set we have the energy function and we shall search for another first integral. 

 Let $D$ be a skew-symmetric derivation on this Lie algebra.  

\begin{enumerate}

\item for $\varepsilon \neq 0, 1$, the map $D$ is trivial ($D\equiv 0$). 

\item For $\varepsilon = -1$, any skew-symmetric derivation satisfies $De_1=- \alpha e_2, De_2=  \alpha e_1$ and $De_5=- \alpha e_6, De_6=  \alpha e_5$, and $De_i=0$ for i=3,4, for $\alpha \in \RR$. 

\item  For $\varepsilon = 0$,  any skew-symmetric derivation satisfies $De_2=- \alpha e_4, De_4=  \alpha e_2$ for $\alpha \in \RR$ and $De_i=0$, for $i=1,3,5,6$.
\end{enumerate}
These computations show that from a skew-symmetric derivation $D$ we cannot have a first integral $f_{D^*}$ that Poisson commute with the first integrals $f_{e_i^*}$, for all $i=3,4,5,6$.

A symmetric map $S:\nn \to \nn$ on this Lie algebra is as follows
\begin{enumerate}
\item For $\varepsilon \neq 0$ $S e_i=\alpha e_i$, for $i=1,2,3,4$, and for the subspace spanned by $e_5, e_6$ one has $S e_5= s_1 e_5 + s_2 e_6$, $Se_6=s_2 e_5 + s_3 e_6$; where $\alpha, s_1, s_2, s_3\in \RR$. Let  $f_{e_5}, f_{e_6}$ the linear first integrals, and let $f_T:=\En -f_{e_5}^2 -f_{e_6}^2$, denote the invariant first integral. Usual computations show that any quadratic first integral, which has no coeficients on the central coordinates, is linear dependent with $f_T$.  

 \item For $\varepsilon=0$, any symmetric map $S$ giving rise to a first integral $g_S$ has the form $S e_i=\alpha e_i$, for $i=1,2,3$, $Se_4=  \alpha e_4$ and $S e_5= s_1 e_5 + s_2 e_6$, $Se_6= s_2 e_5 + s_3 e_6$; where $\alpha, s_1, s_2, s_3\in \RR$. 

And we are not able to get a complete set of first integrals, except for $\varepsilon=-1$, 
\end{enumerate}

\item $n_{6,25}$ has a  abelian Lie subalgebra spanned by $e_2, e_3, e_4, e_5, e_6$ and a complete set of commuting first integrals is induced by the corresponding functions $f_{e_i^*}$, i=2,3,4,5,6 as in Theorem \ref{t2}. 

\item $n_{6,26}$. This is the free 2-step nilpotent Lie algebra in three generators, which was studied in \cite{Ov1}. A complete set is constructed with a symmetric map $S$, derived from the quadratic
$$q(Y)=2(y_1y_6-y_2 y_5+ y_3 y_4), \qquad \mbox{ for } \quad  Y=\sum_{i=1}^6 y_i e_i.$$
One gets a complete set of commuting first integrals  with $\{\En, f_{e_2^*}, f_{e_4}, f_{e_5}, f_{e_6}, g_S\}$, where $g_S(p,Y)=q(Y)$. 
\end{enumerate}

We have proved the next result. 

\begin{prop} With exceptions of the Lie algebras $n_{23}$ and $n_{24}(\varepsilon)$, the rest of $k$-step nilpotent Lie algebras of dimension six, for k=2,3, admits a metric for which there is  a complete set of first integrals. 
\end{prop}

\begin{rem} The other nilpotent Lie algebras in dimension six from the list of de Graaf, are nilpotent of step four or five. In some cases the Lie algebra has a codimension one abelian subalgebra and by Theorem \ref{t2} one gets a complete set of commuting first integrals for any left-invariant metric. 
\end{rem}

\begin{rem} The Lie algebra $n_{6,26}$ is the free 2-step nilpotent Lie algebra in three generators. Butler proved that for any left-invariant metric and any lattice $\Gamma$, the geodesic flow cannot be  Liouville integrable on $T(\Gamma \backslash N_{6,26})$ (see \cite{Bu1}). 
\end{rem}

\end{document}